\documentclass{amsart}
\usepackage{amsfonts, amsmath, amssymb, amsthm, esint}
\usepackage{todonotes}
\numberwithin{equation}{section}

\newcommand{\V}{\Vert}

\newcommand{\RR} {\mathbb R}

\newcommand{\pa} {\partial}
\newcommand{\Cal} {\mathcal}

\newcommand{\beq} {\begin{equation}}
\newcommand{\eeq} {\end{equation}}

\newtheorem{theorem}{Theorem}[section]
\newtheorem{remark}[theorem]{ Remark}

\newtheorem{corollary}[theorem]{Corollary}

\newtheorem{proposition}[theorem]{Proposition}

\newtheorem{lemma}[theorem]{Lemma}

\newtheorem{definition}[theorem]{Definition}

\begin{document}
\title[$L^p$ and $BMO_L$ continuity of multipliers of generalized Laplacians ]{Boundedness of spectral multipliers of generalized Laplacians on compact manifolds with boundary }

\author{Mayukh Mukherjee}
\thanks{The author was partially supported by NSF grant  DMS-1161620.}

\email{mathmukherjee@gmail.com}
\subjclass[2010]{35L05, 47A60, 58J60}
\begin{abstract} 

Consider a second order, strongly elliptic negative semidefinite differential operator $L$ (maybe a system) on a compact Riemannian manifold $\overline{M}$ with smooth boundary, where the domain of $L$ is defined by a coercive boundary condition. %The main result is the $L^\infty - \text{BMO}_L$ 
Classically known results, and also recent work in \cite{DOS} and \cite{DM} establish sufficient conditions for $L^\infty-\text{BMO}_L$ continuity %(which in turn implies $L^p$ continuity) 
of $\varphi(\sqrt{A})$, where $\varphi \in S^0_1(\RR)$, and $A$ is a suitable elliptic operator. %It is a natural question to investigate where such sufficient conditions hold. %
Using a variant of the Cheeger-Gromov-Taylor functional calculus due to \cite{MMV}, and short time bounds on the integral kernel of $e^{tL}$ due to \cite{G}, we prove that a variant of such sufficient conditions holds for our operator $L$. %$\varphi(\sqrt{-L})$. %our operator $L$ . %, %for which Gaussian bounds on the integral kernel of $e^{tL}$ are known only for small time. % and then we use it to establish the $L^p$-continuity of $\varphi(\sqrt{-L})$ for $p \in (1, \infty)$. % Finite propagation speed of $\text{cos }t\sqrt{-L}$ is fundamental to our calculations, after the Cheeger-Gromov-Taylor style of analysis (~\cite{CGT}). Accordingly we try to derive sufficiency conditions on $L$ to ensure the finite propagation speed of $\text{cos }t\sqrt{-L}$.
\end{abstract}
\maketitle
\section{Introduction}
Consider a compact Riemannian manifold $\overline{M}$ with smooth metric and smooth boundary, and a second order strongly elliptic differential operator $L : L^2(M, E) \to L^2(M, E)$ with smooth coefficients, where $E$ is a complex vector bundle with Hermitian metric. %Let the domain of $L$ be $\Cal{D}(L) \subset H^2(M, E)$. Also, a
Assume a regular elliptic boundary condition $B(x, \pa_x)u = 0$ on $\pa M$ (see Proposition 11.9, Chapter 5 of ~\cite{T6}) which makes $L$ into a negative semidefinite self-adjoint operator with domain $\Cal{D}(L) \subset H^2(M, E)$. %We want to explicitly point out that unless otherwise mentioned, our analysis will work when $L$ is a system. 
For notational convenience, we will drop the letter $E$ when denoting the space of sections; for example, $L^2(M)$, or just $L^2$, will stand for $L^2(M, E)$ henceforth. Also, $T : B_1 \to B_2$ will mean that $T$ is a bounded linear operator from $B_1$ to $B_2$, where $B_i$ are Banach spaces.

Given a bounded continuous function $\psi : \RR \longrightarrow \RR$, the spectral theorem %(Theorem \ref{STFSAO} of Appendix B) 
defines 
\beq\label{1.1}
\psi( \sqrt{-L}) : L^2(M) \longrightarrow L^2(M)
\eeq
as a bounded self-adjoint operator. Following ~\cite{T2} and ~\cite{T5}, here we consider functions $\varphi$ in the pseudodifferential function class $S^0_1(\RR)$, which means that 
%$S^0_1(\RR)$ %(see Definition \ref{PDC} in Appendix B). 
%To recall what this means, 
\beq\label{pdc}
\varphi \in S^0_1(\RR) \Longrightarrow |\varphi^{(k)}(\lambda)| \lesssim (1 + |\lambda|)^{-k}, k = 0, 1, 2,...
\eeq

We wish to prove that
\begin{theorem}\label{mainres}
\beq\label{l-b}
\varphi(\sqrt{-L}) : L^\infty \longrightarrow \text{BMO}_L.
\eeq	
\end{theorem}

For the definition of $\text{BMO}_L$, see Definition \ref{4.1} and Lemma \ref{Indep}. 
 As a corollary, we get that 
 \begin{corollary}
 \beq\label{l-l}
 \varphi(\sqrt{-L}) : L^p \longrightarrow L^p, \text{   }\forall \text{  }p \in (1, \infty).
 \eeq
  \end{corollary}
  
  In the absence of a boundary, results like Theorem \ref{mainres} are well-known. Also well-known are sufficient conditions for such results in the general setting of a metric measure space, for example, see Theorem \ref{dm} below, which is due to \cite{DM}. However, except for well-behaved scalar elliptic operators with nice boundary conditions and global ``heat kernel'' bounds, such sufficient conditions are not easy to verify. As mentioned in the abstract, our main technical lemma for proving Theorem \ref{mainres}  is to check the following variant of the sufficiency condition proved in Theorem 3 of \cite{DM} (see also \cite{DY1}):
  \begin{lemma}\label{kl} With $\varphi^{\#}(\sqrt{-L})$ as in (\ref{S2})-(\ref{S3}), denote by $k^{\#}(x, y)$ and $k_t(x, y)$ the integral kernels of $\varphi^{\#}(\sqrt{-L})$ and the composite operator $\varphi^{\#}(\sqrt{-L})e^{tL}$ respectively. We have for some $\varepsilon > 0$, 
  	\beq\label{kilikili}
  	\sup_{t \in (0, \varepsilon]}\sup_{y \in \overline{M}}\int_{M \setminus B_{ \sqrt{t}}(y)}|k^{\#}(x, y) - k_t(x, y)|dx < \infty.
  	\eeq
  	%Also,
  	%\beq
  	%\sup_{t \geq 1}\sup_{y \in M}\int_{d(x, y) \geq \sqrt{t}}|k^{\#}(x, y) - k_t(x, y)|dx < \infty
  	%\eeq 
  \end{lemma}
  
  Observe that the due to the generality of the elliptic operator $L$ (we do not assume, for instance, that $L$ can be written in the form $D^*D$, where $D$ is a first order differential operator), and due to the generality of the boundary conditions, Gaussian bounds on the integral kernel of $e^{tL}$ are rather non-trivial to derive. Short time bounds are known from the work in \cite{G}: 
  \begin{theorem}[Greiner]
  	Let $\overline{M}$ be a compact manifold with boundary, and $L$ be a second order self-adjoint negative semidefinite elliptic system defined by regular elliptic boundary conditions. If $p(t, x, y)$ denotes the integral kernel for $e^{tL}$, for some $\kappa \in (0, \infty)$ we have,
  	\beq\label{g}
  	|p(t, x, y)| \lesssim t^{-n/2}e^{-\kappa d(x, y)^2/t}, t \in (0, 1], x, y \in \overline{M},
  	\eeq
  	
  	and 
  	\beq\label{g1}
  	|\nabla_x p(t, x, y)| \lesssim t^{-n/2 - 1/2}e^{-\kappa d(x, y)^2/t}, t \in (0, 1], x, y \in \overline{M}.
  	\eeq
  \end{theorem}
  
  Since the ``heat kernel'' bounds here are only known for short time, we prove Theorem \ref{mainres} in two main steps. Firstly, we define a concept of local $\text{BMO}_L$ spaces, denoted by $\text{BMO}^\epsilon_L$ (see Definition \ref{4.1} below). Then, Lemma \ref{kl} proves that $\varphi^{\#}(\sqrt{-L}) : L^\infty \to \text{BMO}_L^\epsilon$. This we supplement by the following lemma, which proves that $\text{BMO}_{L}^\epsilon$ is in fact independent of $\epsilon$.
\begin{lemma}\label{Indep}
	\beq
	\V f\V_{\text{BMO}^{\sqrt{2R}}_L} \cong \V f\V_{\text{BMO}^{\sqrt{R}}_L}
	\eeq
	where $R > 0$.
\end{lemma}  
%\subsection{History and motivation}
%Let us take the space to review some literature. 
%In the past, arguably the most well-studied case has been 
\subsection{Tools, preliminaries and motivation}
Our main tool will be the functional calculus used in ~\cite{CGT}, namely
\beq\label{Fc}
\varphi(\sqrt{-L}) = \frac{1}{\sqrt{2\pi}}\int_{-\infty}^\infty \hat{\varphi}(t)e^{it\sqrt{-L}}dt.
\eeq

We see that $\sqrt{-L}$ is a first order elliptic self-adjoint operator with compact resolvent. So, Spec$(\sqrt{-L})$ is a discrete %and real, and since $\sqrt{-L}$ is non-negative, Spec$(\sqrt{-L}) \subset 
subset of $[0, \infty)$,  and it is no loss of generality to assume that $\varphi(\lambda)$ is an even function of $\lambda$. This reduces (\ref{Fc}) to 
\beq\label{fc}
\varphi(\sqrt{-L}) = \frac{1}{\sqrt{2\pi}}\int_{-\infty}^\infty \hat{\varphi}(t)\text{cos }t\sqrt{-L}\text{   }dt,
\eeq
where $\text{cos }t\sqrt{-L}$ is the solution operator to the ``wave equation'' with zero initial velocity, i.e.,
\beq
u(t, x) = \text{cos }t\sqrt{-L}f(x)
\eeq
where
\beq
\pa^2_t u - Lu = 0, \text{   }u(0, x) = f(x), \text{   }\pa_t u(0, x) = 0
\eeq
together with the coercive boundary conditions mentioned above.
%Let us also explain the reason why the finite propagation speed of $\text{cos }t\sqrt{-L}$ is so important. 

We split (\ref{fc}) into two parts in the following way: let, for $a > 0$ small, $\theta (t)$ be an even function, such that
\beq\label{S1}
\theta \in C^\infty_c((-a, a)), \theta (t) \equiv 1 \text{ on }[-a/2, a/2].
\eeq

Now, denote
\beq\label{S2}\hat{\varphi}^{\#}(t) = \theta(t)\hat{\varphi}(t), \text{   }\hat{\varphi}^b(t) = (1 - \theta(t))\hat{\varphi}(t).
\eeq

Then, we can write
\begin{align}\label{S3}
\varphi(\sqrt{-L}) & = \frac{1}{\sqrt{2\pi}}\int_{-\infty}^\infty \hat{\varphi}^{\#}(t)\text{cos }t\sqrt{-L}dt + \frac{1}{\sqrt{2\pi}}\int_{-\infty}^\infty, \hat{\varphi}^b(t)\text{cos }t\sqrt{-L}dt\nonumber\\ 
& = \varphi^{\#}(\sqrt{-L}) + \varphi^b(\sqrt{-L}),
\end{align}
%Also note here that 
where both $\varphi^{\#}(\sqrt{-L})$ and $\varphi^b(\sqrt{-L})$ are self-adjoint. % Indeed, 
%\[
%\varphi^{\#}(\sqrt{-L})^* = \frac{1}{\sqrt{2\pi}}\int_{-\infty}^\infty \hat{\varphi}^{\#}(t)(\text{cos }t\sqrt{-L})^*dt,
%\]
%and the spectral theorem defines $\text{cos }t\sqrt{-L}$ as a bounded self-adjoint operator.\newline

Now, (\ref{pdc}) implies that $\varphi^b$ is smooth and rapidly decreasing. % (to see this, one can use Lemma 1.2, Chapter XII of ~\cite{T2} in conjunction with Fourier inversion). 
So we have
\[
|\varphi^b(\lambda)| \lesssim (1 + |\lambda|)^{-m}, \text{   for some   } m > \frac{n}{2}.
\]

Then, the ellipticity of $L$ implies
\beq\label{1.12}
\varphi^b(\sqrt{-L}) : L^2(M) \longrightarrow H^m(M) \subset C(\overline{M}).
\eeq

This is because, we can write $\varphi^b(\lambda) = (1 + |\lambda|^2)^{-s/2}\psi^b(\lambda), \text{   } s \in \mathbb{N}$, where $\psi^b(\lambda)$ is a bounded function. That implies
\[
\varphi^b(\sqrt{-L}) = (I - L)^{-s/2}\psi^b(\sqrt{-L}).\]

By the spectral theorem, $\psi^b(\sqrt{-L})$ is bounded on $L^2$, and 
\[
(I - L)^{-s/2} : L^2(M) \rightarrow \mathcal{D}((-L)^{s/2}) \subset H^s(M).\]
%Now, the dual of $C(\overline{M})$ is all finite regular Borel measures (see ~\cite{Ru}, Theorem 6.19), and,  %, and $L^1$ is a subset of it. In fact, $L^1$ is the closure of $L^2$ in it. This gives

We can see that $L^1(M) \subset C(\overline{M})^*$ and the inclusion is continuous. Also, when $m > n/2$, $H^m(M) \subset C(\overline{M})$ by Sobolev embedding. This gives, via duality on (\ref{1.12}), and self-adjointness of $\varphi^b(\sqrt{-L})$, that 
\beq\label{be4}
\varphi^b(\sqrt{-L}) : L^1(M) \longrightarrow L^2(M).
\eeq 

%is bounded.
This %, observing the fact that $\overline{M}$ is compact and hence $\V .\V_{L^p} \lesssim \V .\V_{L^q}$ when $p \leq q$, 
interpolates with (\ref{1.1}) to give %the continuity of 
\beq
\varphi^b(\sqrt{-L}) : L^p(M) \longrightarrow L^p(M) \text{   } \forall \text{  } p \in (1, \infty).
\eeq

Results of the type (\ref{l-l}) have been well-studied for complete Riemannian manifolds $M$ without boundary and $L = \Delta$, the Laplace-Beltrami operator. When $M$ is compact, we refer to ~\cite{T1}, particularly the combination of Theorem 1.3 of Chapter XII and Theorem 2.5 of Chapter XI. %Even in the non-compact setting, very definitive results abound. For more details, \
For results in the non-compact setting, refer to ~\cite{CGT}, ~\cite{MMV}, \cite{T3}, ~\cite{T4}, ~\cite{T5}, etc. In the papers which deal with a non-compact setting, an additional difficulty is in analyzing $\varphi^b(\sqrt{-\Delta})$, because of the failure of the compact Sobolev embedding. Particularly for manifolds like the hyperbolic space, where the volume growth is exponential with respect to the distance, %of a radius $r$-ball grows exponentially with $r$, 
one requires more stringent restrictions on $\varphi$, namely, $\varphi$ being holomorphic on a strip around the $x$-axis, satisfying bounds of the form (\ref{pdc}). This condition was first introduced in the paper ~\cite{CS}. This motivated some research on the optimal width of said strip; %that allowed the $L^p$ boundedness results to go through on a non-compact setting. T
to the best of our knowledge, the sharpest results in this direction appear in ~\cite{T4}. Also, in ~\cite{CGT}, $\varphi^{\#}(\sqrt{-\Delta})$ was analyzed as a pseudodifferential operator, something we cannot do in our present setting because of the presence of a boundary. %as even the square root of the Dirichlet Laplacian does not admit interpretation as a pseudodifferential operator (
It is well-known that the scalar square root of the Laplacian fails the ``transmission condition'' (see, for example, equation (18.2.20) of ~\cite{Ho}) to be a pseudodifferential operator. Our analysis will be a combination of methods from ~\cite{DOS} and ~\cite{DM}, and the approach of ~\cite{T5}, which follows and refines the results in ~\cite{MMV}. % on the other. 

Let us discuss some of the main lines of investigation in ~\cite{DOS} and ~\cite{DM}. The results therein are rather general, set in the context of an open subset $X$ of a metric measure space of homogeneous type. If $L$ is a negative semi-definite, self-adjoint operator on $L^2(X)$, they assume that the integral kernel of $L$ satisfies
\beq\label{ass}
|p(t, x, y)| \lesssim t^{-n/m}e^{-\kappa\text{dist }(x, y)^{m/(m - 1)}/t^{1/(m - 1)}}, \text{  } 0 < t \leq 1, \kappa \in (0, \infty).
\eeq
%The motivation behind such assumption is understandable in view of (\ref{g}) and (\ref{g1}). 

Theorem 3.1 of ~\cite{DOS} establishes $L^p$-continuity of 
$\varphi((-L)^{1/m})$ via proving that it is weak type $(1, 1)$. What mainly concerns us with ~\cite{DOS} is the fact that they proved $L^p$ continuity of $\varphi((-L)^{1/m})$ via using the following result from ~\cite{DM}:
\begin{theorem}[Duong-McIntosh]\label{dm} Under the hypothesis on $L$ outlined above, let $k_t(x, y)$ denote the integral kernel of $\varphi((-L)^{1/m})(I - e^{tL})$, where $\varphi : \RR \to \RR$ is bounded and continuous. Also assume
	\beq\label{BIGASS}
	\sup_{t > 0}\sup_{y \in X} \int_{X \setminus B_{t^{1/m}}(y)} |k_t(x, y)| dx < \infty.
	\eeq
	Then $\varphi((-L)^{1/m})$ is of weak type $(1, 1)$.
\end{theorem}

It is naturally interesting to investigate when conditions like (\ref{BIGASS}) are satisfied. In this paper, our aim is to check that (\ref{BIGASS}) holds for a large class of operators $L$, as mentioned before, in the setting of a smooth compact manifold with boundary. %N %ote that, in ~\cite{DOS}, (\ref{BIGASS}) is only an assumption. To the best of our knowledge, Lemma \ref{kl} is the first attempt to prove variants of this assumption in explicit situations. It is not too surprising that ~\cite{DM}, working on metric measure spaces, has to assume (\ref{BIGASS}), because our proof of Lemma \ref{kl} here is intrinsically connected to the fact that $L$ is an elliptic differential operator, and the fact that the setting is on a smooth manifold. It seems interesting to investigate whether one can do away with all the analytic machinery such a setting provides and prove a variant of (\ref{BIGASS}) in the general setting of a metric measure space. %\newline
%Let us denote by $p(t, x, y)$ the integral kernel of $e^{tL}$, so that
%\[
%e^{tL}f(x) = \int_{M} p(t, x, y)f(y) dy.
%\]
%Henceforth, by straightforward analogy, we will refer to $p(t, x, y)$ as the heat kernel. Amongst our most important tools will be the following estimates in ~\cite{G}:
%\begin{proposition} (Greiner)
%	Let $\overline{M}$ be a compact manifold with boundary, and $L$ be a second order self-adjoint nenative semidefinite elliptic system defined by regular elliptic boundary conditions. For some $\kappa \in (0, \infty)$, we have
%	\beq\label{g}
%	|p(t, x, y)| \lesssim t^{-n/2}e^{-\kappa d(x, y)^2/t}, t \in (0, 1], x, y \in \overline{M},
%	\eeq
%	and 
%	\beq\label{g1}
%	|\nabla_x p(t, x, y)| \lesssim t^{-n/2 - 1/2}e^{-\kappa d(x, y)^2/t}, t \in (0, 1], x, y \in \overline{M}.
%	\eeq
%\end{proposition}

We also note here that in the proof of $L^p$-boundedness of $\varphi(\sqrt{-L})$ in ~\cite{T2}, the main approach is to prove the following 
\begin{lemma}[Taylor]
	There exists $C < \infty$, independent of $s \in (0, 1]$ and of $y, y' \in \overline{M}$, such that 
	\[
	\text{dist } (y, y') \leq \frac{s}{2} \implies \V K^{\#}(., y) - K^{\#}(., y')\V_{L^1(B_1(y) \setminus B_s(y))} \leq C,
	\]
	where $K^{\#}(x, y)$ is the integral kernel of $\varphi^{\#}(\sqrt{-L})$.
\end{lemma}
With that in place, as is noted in ~\cite{T2}, the weak type $(1, 1)$ property of $\varphi^{\#}(\sqrt{-L})$ is a consequence of Proposition 3.1 of ~\cite{MMV}, which is a variant of Theorem 2.4 in Chapter III of ~\cite{CW}.
%Ultimately we establish that 
\
%\beq
%\varphi(\sqrt{-L}) : \text{BMO}_L \longrightarrow \text{BMO}_L
%\eeq

Note that we are yet to argue the $L^\infty - \text{BMO}_L$ boundedness of $\varphi^b(\sqrt{-L})$; this we will do at the beginning of Section \ref{3.6}. %with similar arguments. 
So, for all continuity related aspects, for the rest of our investigation, we will mainly be concerned with just $\varphi^{\#}(\sqrt{-L})$. These boundedness considerations will largely be addressed in Sections \ref{BMOD} and \ref{3.6}. For those sections, our standing assumptions will be the following:\newline
{\bf Assumption:}\label{3} $\text{cos }t\sqrt{-L}$ has finite speed of propagation, which, by scaling $L$ if necessary, we will assume to be $\leq 1$.

It seems a challenging question to determine when $\text{cos }t\sqrt{-L}$ has finite speed of propagation. However, for a reasonable class of operators $L$, we can prove the following ``Davies-Gaffney'' type estimates:
\begin{proposition}\label{FSOP}
	Let $-L = D^*D + H$ with the generalized Dirichlet or Neumann boundary conditions on $D$, as defined in (\ref{DC}) and (\ref{NC}), and with $H \geq 0$ in $L^2(M)$. Let $U, V$ be two open balls such that dist$(U, V) = r$. With $t>0$ fixed, let $\phi(x) = \frac{r}{t}\text{dist }(x, U)$ and $P = [D, e^\phi]$ denote the usual commutator operator. Then $\text{cos }t\sqrt{-L}$ has finite speed of propagation if it satisfies for all $v \in L^2(V)$ the following:
	\beq\label{cond}
	\V e^{-\phi /2}Pv\V_{L^2} \leq \frac{r}{t}\V e^{\phi /2}v\V_{L^2}.
	\eeq
\end{proposition}

(\ref{cond}) follows when $|D\phi|$ is bounded. As a special case, (\ref{cond}) follows trivially when $L = \Delta$ with the Dirichlet or Neumann boundary conditions. 
%Assumption 2: $L$ has nice scaling properties, in the following sense:\newline
%let $L_x$ denote the operator $L$ in an open ball $U \subset \overline{M}$ in local coordinates $x$. Now, scale the metric in the ball $U$ by a factor of $r$. Let the new coordinates be given by $\tilde{x}$ and let $L_{\tilde{x}}$ be the operator $L$ in the new coordinates. Let $\varphi (x)$ be a function in $C^\infty_c(U)$ in the $x$ coordinates and let $\tilde{\varphi}(\tilde{x})$ be the corresponding function in the $\tilde{x}$ coordinates. Then, we must have 
%\beq\label{ASS}
%\tilde{L}\tilde{\varphi} = \frac{1}{r^2}L\varphi.
%\eeq
   %provided we have the finite propagation speed of 
%$\text{cos }t\sqrt{-L}$.\newline
%\subsection{Main tools}
%Let us denote by $p(t, x, y)$ the integral kernel of $e^{tL}$, so that
%\[
%e^{tL}f(x) = \int_{M} p(t, x, y)f(y) dy.
%\]
%Henceforth, by straightforward analogy, we will refer to $p(t, x, y)$ as the heat kernel. Amongst our most important tools will be the following estimates in ~\cite{G}:
%\begin{proposition} (Greiner)
%Let $\overline{M}$ be a compact manifold with boundary, and $L$ be a second order self-adjoint nenative semidefinite elliptic system defined by regular elliptic boundary conditions. For some $\kappa \in (0, \infty)$, we have
%\beq\label{g}
%|p(t, x, y)| \lesssim t^{-n/2}e^{-\kappa d(x, y)^2/t}, t \in (0, 1], x, y \in \overline{M},
%\eeq
%and 
%\beq\label{g1}
%|\nabla_x p(t, x, y)| \lesssim t^{-n/2 - 1/2}e^{-\kappa d(x, y)^2/t}, t \in (0, 1], x, y %\in \overline{M}.
%\eeq
%\end{proposition}

One last comment: %we will use (\ref{g}) and (\ref{g1}) to derive some technical estimates on the heat and Poisson semigroups (see Appendix $1$), which will be fundamental to the boundedness estimates we seek to establish. Also, 
for the purposes of proving Lemma \ref{kl}, we will also %use (\ref{Fc}) directly, but use 
a modification of (\ref{Fc}), following ~\cite{MMV}, ~\cite{T5} and ~\cite{T2}. For details on this, see Subsection \ref{modi}. 
% We note here that ~\cite{G} has to use a theory of parabolic layer potentials because of the existence of a non-trivial boundary.\newline

\subsection{Outline of the paper}
%Here, we take the space to describe the overall outline of the present paper. 
 In Section \ref{BMOD}, we prove that Definition \ref{4.1} of $\text{BMO}^\epsilon_L$ is independent of $\epsilon$, as long as we are on a compact setting. This is the content of Lemma \ref{Indep}. Then we proceed to prove our main technical lemma of the paper, Lemma \ref{kl}. We begin Section \ref{3.6} by arguing the $L^\infty - \text{BMO}_L$ continuity of $\varphi^b(\sqrt{-L})$, and then prove in Proposition \ref{4.4} (using Lemma \ref{kl}) the $L^\infty-\text{BMO}_L$ continuity of $\varphi^{\#}(\sqrt{-L})$, which finally proves Theorem \ref{mainres}. In Appendix 1, we collect together some useful information about the integral kernels of the operators $e^{tL}$ and $e^{-t\sqrt{-L}}$. The properties we establish are quite parallel to their usual scalar Laplacian counterparts, and are at the background of some of the estimates we derive in the main body of the paper. In Appendix 2, we prove some partial results towards establishing sufficient criteria for $\text{cos }t\sqrt{-L}$ to have finite propagation speed. To wit, we prove that for those operators $L$ which can be written in the specific form (\ref{special}), under generalized Dirichlet or Neumann boundary conditions (see (\ref{DC}) and (\ref{NC}) below), and under  the assumption (\ref{cond}), $\text{cos }t\sqrt{-L}$ has finite speed of propagation. This is the content of Proposition \ref{FSOP}.

\section{Proof of Lemma \ref{kl}}\label{BMOD}
\subsection{$\text{BMO}_L$ and its variants}
%In this section we investigate the continuity property of $\varphi(\sqrt{-L})$ from $L^\infty$ to $BMO_L$-spaces, under the standing assumptions of finite speed of propagation of $\text{cos  }t\sqrt{-L}$. %and the scaling property of $L$ mentioned in the introduction. %It is known that the $L^p$ continuity results for even $\varphi(\sqrt{-\Delta})$ cannot be extended to the endpoint cases $L^1$ and $L^\infty$ (WHY ???). 
%(WHY???)
%So if we want to have an extension at all, we must use BMO type spaces respectively as the replacement for $L^\infty$ (SHOULD ONE JUSTIFY THIS??). \newline
We combine the definition of BMO$_L$ in ~\cite{DY} with the definition of local BMO spaces in ~\cite{T3} to 
give the following
\begin{definition}\label{4.1}
$f \in L^1_{loc}(M)$ is in BMO$^\epsilon_L(M)$ if 
\beq
\frac{1}{|B|}\int_B |f(x) - e^{tL}f(x)|dx \leq C,
\eeq
where $\sqrt{t}$ is the radius of the ball $B$, and $B$ ranges over all balls in $M$ of radius $\leq \epsilon$. Let
\beq
\V f\V_{\text{BMO}^\epsilon_L} = \sup_{B \in \Cal{B}}\frac{1}{|B|}\int_B |f(x) - e^{tL}f(x)|dx,
\eeq
where $\sqrt{t}$ is the radius of the ball $B$, and $\Cal{B}$ contains all balls of radius $\leq \epsilon$.
% where $\epsilon$ satisfies $\frac{1}{2\sqrt{\epsilon}} \geq \mbox{diam} M$. 
\end{definition}
We now make the observation that our definition of $BMO_L^\epsilon$ is actually independent of the $\epsilon$ chosen. %. However, the reason for the restriction on $\epsilon$ in the definition (namely, $\frac{1}{2\sqrt{\epsilon}} \geq \mbox{diam} M$) will be made clear in the sequel. Let us sketch out the following 

\begin{proof}
Clearly, 
\[
\V f\V_{\text{BMO}^{\sqrt{2R}}_L} \geq \V f\V_{\text{BMO}^{\sqrt{R}}_L}.\]
For the reverse inequality, let us fix a point $y \in \overline{M}$. Then we have, for $r \leq R$,
\begin{align*}
\frac{1}{|B_{\sqrt{2r}}(y)|}\int_{B_{\sqrt{2r}}(y)}|f(x) - e^{2rL}f(x)|dx - \frac{1}{|B_{\sqrt{r}}(y)|}\int_{B_{\sqrt{r}}(y)}|f(x) - e^{rL}f(x)|dx \\ \lesssim \frac{1}{|B_{\sqrt{r}}(y)|}\int_{B_{\sqrt{2r}}(y)}|f(x) - e^{2rL}f(x) - \chi_{B_{\sqrt{r}}(y)}(x)f(x) + \chi_{B_{\sqrt{r}}(y)}(x)e^{rL}f(x)|dx\\
= \frac{1}{|B_{\sqrt{r}}(y)|}\int_{B_{\sqrt{2r}}(y)}|\chi_A(x) f(x) - e^{2rL}f(x) + \chi_{B_{\sqrt{r}}(y)}(x) e^{rL}f(x)|dx,
\end{align*}
where $A$ denotes the ``annulus'' $B_{\sqrt{2r}}(y) \setminus B_{\sqrt{r}}(y)$, which can be covered by at most $K$ balls of radius $\sqrt{r}$, where $K$ is a positive number independent of $r$, because $\overline{M}$ is compact. Also, in the ensuing calculation, we tacitly use the fact that the volume of a ball of radius $r$ is uniformly bounded. \newline

Now, the last quantity in the above equation is 
\[
\leq \frac{1}{|B_{\sqrt{r}}(y)|}\int_{B_{\sqrt{2r}}(y)}|\chi_A(x) f(x) - \chi_A(x) e^{2rL}f(x)|dx \]\[ + \frac{1}{|B_{\sqrt{r}}(y)|}\int_{B_{\sqrt{2r}}(y)}|\chi_{B_{\sqrt{r}}(y)}(x) e^{2rL}f(x) - \chi_{B_{\sqrt{r}}(y)}(x) e^{rL}f(x)|dx 
\]
\[
\leq \frac{1}{|B_{\sqrt{r}}(y)|}\int_{B_{\sqrt{2r}}(y)}|\chi_A(x) f(x) - \chi_A(x) e^{rL}f(x)|dx \]
\[+ \frac{1}{|B_{\sqrt{r}}(y)|}\int_{B_{\sqrt{2r}}(y)}|\chi_A(x) e^{2rL}f(x) - \chi_A(x) e^{rL}f(x)|dx 
 \]
 \[
+ \frac{1}{|B_{\sqrt{r}}(y)|}\int_{B_{\sqrt{2r}}(y)}|\chi_{B_{\sqrt{r}}(y)}(x) e^{2rL}f(x) - \chi_{B_{\sqrt{r}}(y)}(x) e^{rL}f(x)|dx
\]
\[
= \frac{1}{|B_{\sqrt{r}}(y)|}\int_A|f(x) - e^{rL}f(x)|dx + \frac{1}{|B_{\sqrt{r}}(y)|}\int_A|e^{2rL}f(x) - e^{rL}f(x)|dx  
\]
\[
+ \frac{1}{|B_{\sqrt{r}}(y)|}\int_{B_{\sqrt{r}}(y)}|e^{2rL}f(x) - e^{rL}f(x)|dx\\
= A + B + C \text{  (say)  }.
\]

Putting everything together, we have that
\begin{align}\label{seshlala}
\frac{1}{|B_{\sqrt{2r}}(y)|}\int_{B_{\sqrt{2r}}(y)}|f(x) & - e^{2rL}f(x)|dx - \frac{1}{|B_{\sqrt{r}}(y)|}\int_{B_{\sqrt{r}}(y)}|f(x) - e^{rL}f(x)|dx \\
\lesssim A + B + C. \nonumber
\end{align}

Now, if we let $e^{rL}f(x) = g(x)$, and can prove that 
\beq\label{jalalema}
\V g\V_{\text{BMO}^{\sqrt{r}}_L} \lesssim \V f\V_{\text{BMO}^{\sqrt{r}}_L},
\eeq
then we have that each of $A, B, C$ is $\lesssim \V f\V_{\text{BMO}^{\sqrt{R}}_L}$, giving us our result from (\ref{seshlala}).\newline
%To prove (\ref{jalalema}), we see that we have to prove
%\[
%\V e^{rL}(e^{rL}f - f)\V_{L^1(B_{\sqrt{r}}(y))} \lesssim \V e^{rL}f - f\V_{L^1(B_{\sqrt{r}}(y))}\]
%which follows from the following:
%\[
%\V e^{rL}u\V_{L^1(B_{\sqrt{r}}(y))} = \int_{B_{\sqrt{r}}(y)} |e^{rL}%u(x)|dx \leq \int_{B_{\sqrt{r}}(y)}\int_M |p(r, x, z)u(z)|dzdx\]
%\[ \lesssim r^{-n/2}\int\V f\V_{L^1(B_{\sqrt{r}}(y))}dx \lesssim r^{-%n/2}\V f\V_{L^1(B_{\sqrt{r}}(y))}.\]
%\end{proof}
%\begin{remark}
%(SHOULD BE DELETED??) A doubling property (see Definition (\ref{DP}) in Appendix A) is the main driving force behind the proof of the above lemma. However, since we are working in a compact setting, we have not entered into the intricacies of a proof that would work in more general settings. 
%\end{remark}

Now we justify (\ref{jalalema}). Choose $\varepsilon > 0$ and let $\Cal{B}$ consist of all balls in $M$ whose radii are $\leq \varepsilon$. Choose $s > 0$ such that $\sqrt{s} \leq \varepsilon$. As per the notation above, if $u(x) = e^{sL}f(x) - f(x)$, observe that it suffices to prove that 
\[
\sup_{B_{\sqrt{s}} \in \Cal{B}}\frac{1}{|B_{\sqrt{s}}|}\int_{B_{\sqrt{s}}}|e^{sL}u(x)|dx \lesssim \sup_{B_{\sqrt{s}} \in \Cal{B}}\frac{1}{|B_{\sqrt{s}}|}\int_{B_{\sqrt{s}}}|u(x)|dx.
\]

Now, we have,
\begin{align*}
\sup_{B_{\sqrt{s}} \in \Cal{B}}\frac{1}{|B_{\sqrt{s}}|}\int_{B_{\sqrt{s}}}|e^{sL}u(x)|dx & \lesssim \sup_{B_{\sqrt{s}} \in \Cal{B}}\frac{1}{|B_{\sqrt{s}}|}\int_{B_{\sqrt{s}}}\int_M |p(s, z, x)u(z)|dzdx\\
& \lesssim \sup_{B_{\sqrt{s}} \in \Cal{B}}\frac{1}{|B_{\sqrt{s}}|}s^{-n/2}\int_{B_{\sqrt{s}}}\V u\V_{L^1(M)}\text{  (from (\ref{g}))}\\
& = s^{-n/2}\V u\V_{L^1(M)}.
\end{align*}

So we are done if we can prove that 
\[
\V u\V_{L^1(M)} \lesssim \sup_{B_{\sqrt{s}} \in \Cal{B}}\frac{1}{|B_{\sqrt{s}}|}\V u\V_{L^1(B_{\sqrt{s}})}.\]
%This will be proved if we can prove that  
%\beq\label{certainlytrue}
%\frac{1}{|M|}\V u\V_{L^1(M)} \leq \sup_{B_{\sqrt{s}} \in \Cal{B}}\frac{1}{|B_{\sqrt{s}}|}\V u\V_{L^1(B_{\sqrt{s}})},
%\eeq
%which is certainly true. To derive a rigorous demonstration of (\ref{certainlytrue}), one can 

Consider a partition of $\overline{M}$ into balls coming from $\Cal{B}$. Let\label{sym19} $\overline{M} =  \coprod_n B_n$, where $B_n \in \Cal{B}$. Then,
\begin{align*}
\frac{1}{|M|}\V u\V_{L^1(M)} & = \frac{1}{|M|}\sum_n \V u\V_{L^1(B_n)} = \frac{1}{|M|}\sum_n |B_n|\frac{1}{|B_n|}\V u\V_{L^1(B_n)}\\
& \leq \frac{1}{|M|} \sum_n |B_n| \sup_{B \in \Cal{B}}\frac{1}{|B|}\V u\V_{L^1(B)}  = \sup_{B \in \Cal{B}}\frac{1}{|B|}\V u\V_{L^1(B)}.
\end{align*}

This finishes the proof.
\end{proof}
\subsection{A modification of the \cite{CGT} functional calculus}\label{modi}
At this point, let us recall the main approach of ~\cite{MMV}, ~\cite{T5} and ~\cite{T2}. The analysis in these papers avoided producing a parametrix for (\ref{fc}). Instead, they replaced (\ref{fc}) by the following
\beq\label{fc1}
\varphi(\sqrt{-L}) = \frac{1}{2} \int^\infty_{-\infty}\varphi_k(t)\Cal{J}_{k - 1/2}(t\sqrt{-L})\text{   }dt,
\eeq
where 
\[
\Cal{J}_\nu(\lambda) = \lambda^{-\nu}J_\nu(\lambda),
\]
$J_\nu(\lambda)$ denoting the standard Bessel function (see ~\cite{T5}, equation (3.1) and ~\cite{T6}, Chapter 3, Section 6 for more details on Bessel functions), and
\[
\varphi_k(t) = \prod^{k}_{j = 1}(-t\frac{d}{dt} + 2j - 2)\hat{\varphi}(t).
\]

\cite{T5} derives (\ref{fc1}) from (\ref{fc}) by an integration by parts argument (see (3.7) - (3.9) of ~\cite{T5}). 
Similarly, from ~\cite{T5}, (3.14), we have
\beq\label{B}
\varphi^{\#}(\sqrt{-L}) = \frac{1}{2} \int^\infty_{-\infty}\psi_k(t)\Cal{J}_{k - 1/2}(t\sqrt{-L})\text{  }dt
\eeq
with 
\beq\label{S11}
\psi_k(t) = \prod^k_{j = 1}(-t\frac{d}{dt} + 2j - 2)\hat{\varphi}^{\#}(t), \text{ where } \text{supp  }\psi_k \subset [-a, a]. 
\eeq
%where,
%\beq\label{S}
%\text{supp  }\psi_k \subset [-a, a].
%\eeq

Also, (\ref{pdc}) implies %(see ~\cite{T6}, (1.23))
\[
|(t\pa_t)^j\hat{\varphi}(t)| \leq C_j|t|^{-1}, \text{   } \forall j \in \{0, 1,...,[\frac{n}{2}] + 2\},
\]

which in turn implies 
\beq\label{boddopr}
|\psi_k(t)| \leq C_k |t|^{-1}, \text{   } 0 \leq k \leq [\frac{n}{2}] + 2.
\eeq

Now, let $k^{\#}(x, y)$ denote the integral kernel of $\varphi^{\#}(\sqrt{-L})$, that is,
\beq\label{chherede}
\varphi^{\#}(\sqrt{-L})f(x) = \int_M k^{\#}(x, y)f(y)dy.
\eeq

Without any loss of generality, we can scale $L$ so that the speed of propagation of $\text{cos  }t\sqrt{-L}$ is $\leq 1$, which has been stated as an assumption on page \pageref{3}. Also, let us select $a = 1$ in (\ref{S11}) and (\ref{S1}).

Now we address one fundamental question: why this choice of $a$ and why is the finite propagation speed of $\text{cos }t\sqrt{-L}$ so important? This has to do with the support of the integral kernel $k^{\#}(x, y)$. If the speed of propagation of  $\text{cos  }t\sqrt{-L}$ is $\leq 1$, then $k^{\#}(x, y)$ is supported within a distance $\leq 1$ from the diagonal. %\footnote{For a definition of the diagonal, see Definition \ref{DIAG} in Appendix D.}. 
Let us justify this: we have
\[
\varphi^{\#}(\sqrt{-L})f(x) = \int_M k^{\#}(x, y)f(y)dy = \frac{1}{\sqrt{2\pi}}\int^\infty_{-\infty} \hat{\varphi}^{\#} (t) \text{cos }t\sqrt{-L}f(x)dt.
\]

Suppose the propagation speed of $\text{cos }t\sqrt{-L}$ is $\leq 1$. Then, when $|t| \leq 1$, 
\[
\text{supp }\text{cos }t\sqrt{-L}f(x) \subset \{ x \in \overline{M} : \text{dist }(x, \text{supp }f) \leq t\}.\] 

When $|t| > 1$, $\hat{\varphi}^{\#}(t) = 0$. So, for all $t \in \RR$, $\varphi^{\#}(\sqrt{-L})f(x)$ will be zero for all $x \in \overline{M}$ such that dist$(x, \text{supp }f) > 1$. Since this happens for all $f \in L^2(M)$, we have that $k^{\#}(x, y)$ is supported within a distance of 1 from the diagonal. This property will be crucially used in the sequel.

Using (\ref{B}), (\ref{S11}) and the fact that $\psi_k$ is an even function, we can write
\beq\label{hihi}
k^{\#}(x, z) = \int^1_0\psi_k(s)B_k(s, x, z)ds
\eeq
where $B_k(t, x, y)$ is the integral kernel of $\Cal{J}_{k - 1/2}(t\sqrt{-L})$, that is,
\[
\Cal{J}_{k - 1/2}(t\sqrt{-L})f(x) = \int_M B_k(t, x, y)f(y) dy.
\]

Let us also record the following formula: % in ~\cite{T6}, (1.27):
\beq\label{arektajalale}
\Cal{J}_{k - 1/2}(t\sqrt{-L}) \approx \int^1_{-1} (1 - s^2)^{k - 1}\text{cos } st\sqrt{-L} ds.
\eeq

 We have assumed that the speed of propagation of $\text{cos }t\sqrt{-L}$ is $\leq 1$. Observe that, written in symbols, this means, 
\[
\text{supp } f \subset K \Rightarrow \text{supp }\text{cos }t\sqrt{-L}f \subset K_{|t|},
\]
where $K_{|t|} = \{x \in \overline{M} : \text{dist}(x, K) \leq |t|\}$. This gives, in conjunction with (\ref{arektajalale}),
\beq\label{seshmuhurto}
\text{supp } f \subset K \Rightarrow \text{supp }\Cal{J}_{k - 1/2}(t\sqrt{-L})f \subset K_{|t|}.
\eeq

We now derive a technical estimate on $\V B_k(s, x, .)\V_{L^2(B_1(x)})$ which we will find essential in the sequel. These estimates are variants of Lemma 2.2 in ~\cite{T2}.
\begin{lemma}\label{label}
	 If $G : \RR \longrightarrow \RR$ satisfies
	 \beq
	 |G(\lambda)| \lesssim (1 + |\lambda|)^{-\gamma - 1}, \gamma > n/2,
	 \eeq
	 
	 then 
	 \beq\label{G}
	 \V G(s\sqrt{-L})\V_{\Cal{L}(L^2, L^\infty)} \lesssim s^{-n/2}, s \in (0, 1].
	 \eeq
	 
	 %and
	 %\beq\label{Sting1}
	% \V LG(s\sqrt{-L})\V_{\mathcal{L}(L^2, L^\infty)} \lesssim s^{-n/2 -2}, s \in (0, 1].
	 %\eeq
	 
	 This implies, in particular, the following
\beq 
\V B_k(s, x, .)\V_{L^2(B_1(x) )} \lesssim s^{-n/2}, \text{      } s \in (0, 1], x \in \overline{M}.
\eeq
\end{lemma} 
\begin{proof}
%We will first prove that if $G : \RR \longrightarrow \RR$ satisfies
%\beq
%|G(\lambda)| \lesssim (1 + |\lambda|)^{-\gamma - 1}, \gamma > n/2,
%\eeq
%then 
%\beq\label{G}
%\V G(s\sqrt{-L})\V_{\Cal{L}(L^2, L^\infty)} \lesssim s^{-n/2}, s \in (0, 1].
%\eeq
%Why is this more general? Because 
%From (\ref{B}) and (\ref{4.6}), we have 
%\[
%\Cal{J}_{k - 1/2}(s\sqrt{-L})f = \int_M B_k(s, x, z) f(z) dz.
%\]

We use the following estimate: 
% proved in ~\cite{T6} (see equations (1.25) and (1.26)):
\[
|\Cal{J}_{k - 1/2}(\lambda)| \lesssim (1 + |\lambda|)^{-k}, k > 0.
\]

We write
\beq\label{w}
G(s\sqrt{-L}) =  (I - s^2L)^{-\sigma}G(s\sqrt{-L})(I - s^2L)^\sigma, 2\sigma = \gamma + 1.
\eeq

Now, let $F(\lambda) = G(\lambda)(1 - \lambda^2)^{\sigma}$. Then, using $\gamma > n/2$ and $2\sigma = \gamma + 1$, we see that
\begin{align*}
|F(\lambda)| & = |G(\lambda)(1 - \lambda^2)^{\sigma}| \lesssim (1 + |\lambda|)^{-\gamma - 1}|1 - \lambda^2|^\sigma\\
& \lesssim (1 + |\lambda|)^{2\sigma - \gamma - 1} \leq C.
\end{align*}

Since $F$ is bounded, by the spectral theorem, $F(s\sqrt{-L}) : L^2 \longrightarrow L^2$ is continuous. So by virtue of (\ref{w}), our task is reduced to proving that 
\beq\label{red}
\V (I - s^2L)^{-\sigma}\V_{\Cal{L}(L^2, L^\infty)} \lesssim s^{-n/2}, \sigma > n/4.
\eeq

Now, we use the following identity from that can be derived from the definition of the gamma function: %~\cite{T6} (see equation (2.16) of ~\cite{T6}):
\[
(I - s^2L)^{-\sigma} \approx   \int^\infty_0 e^{-r}e^{rs^2L}r^{\sigma - 1}dr,
\]
which gives
\begin{align*}
\V (I - s^2L)^{-\sigma}\V_{\Cal{L}(L^2, L^\infty)} & \lesssim \int_0^{s^{-2}} e^{-r}\V e^{rs^2L}\V_{\Cal{L}(L^2, L^\infty)}r^{\sigma - 1}dr + \int_{s^{-2}}^\infty e^{-r}\V e^{rs^2L}\V_{\Cal{L}(L^2, L^\infty)}r^{\sigma - 1}dr\\
& \lesssim \int^{s^{-2}}_0 e^{-r}(rs^2)^{-n/4}r^{\sigma - 1}dr + \int_{s^{-2}}^\infty e^{-r}r^{\sigma - 1}dr\\
& \lesssim (s^{-n/2} + 1) \lesssim s^{-n/2}, \text{  } s \in (0, 1],
\end{align*}
where in going from the second to the third step, we have used that 
\[\V e^{tL}\V_{\Cal{L}(L^2, L^\infty)} \lesssim t^{-n/4}, \text{  } t \in (0, 1]
\]
and
\[\V e^{tL}\V_{\Cal{L}(L^2, L^\infty)} \lesssim 1, \text{  } t > 1.\]

This establishes (\ref{G}). That is,
\[
|G(s\sqrt{-L}) f(x)| \lesssim s^{-n/2}\V f\V_{L^2}, \text{  } s \in (0, 1].
\]

In particular, with $f = \delta_x$ and using the compactness of $\overline{M}$, we have
\[
\V g(s, x, .)\V_{L^2} \lesssim s^{-n/2}, \text{  } s \in (0, 1].
\]
\end{proof}
%With that in place, now, let $k^{\#}(x, y)$ be the integral kernel for $\varphi^{\#}(\sqrt{-L})$. Let us define
%\beq
%p(t, z, y) = \chi_{B_{\frac{1}{2\sqrt{t}}(y)}}(z)h(t, z, y)
%\eeq
%where we now use the notation $h(t, x, y)$ for the heat kernel of $e^{tL}$, $\chi$ represents, as usual, the characteristic function of a set, and $B_r(x)$ denotes the ball of radius $r$ around the point $x$. Let $A_t$ denote the operator generated by $p(t, x, y)$ and $k_t(x, y)$ denote the integral kernel for the composite operator $\varphi^{\#}(\sqrt{-L})A_t$, the subscript $t$ denoting the time dependence of the operator $\varphi^{\#}(\sqrt{-L})A_t$.  \newline
 
We are in a position to prove Lemma \ref{kl}.

\begin{proof}
We have
\beq\label{kukur}
\varphi^{\#}(\sqrt{-L})e^{tL}f(x) = \int_M k_t(x, y)f(y) dy.
\eeq

Also,
\begin{align*}
\varphi^{\#}(\sqrt{-L})e^{tL}f(x) & = \int_M k^{\#}(x, y)e^{tL}f(y)dy
 = \int_M k^{\#}(x, y)\int_M p(t, y, z) f(z) dzdy\\
& = \int_M\int_M k^{\#}(x, y)p(t, y, z) f(z) dzdy.
\end{align*}

Interchanging the variables $y$ and $z$, we get
\beq\label{beral}
\varphi^{\#}(\sqrt{-L})e^{tL}f(x) = \int_M \int_M k^{\#}(x, z)p(t, z, y) f(y) dydz.
\eeq

Comparing (\ref{kukur}) and (\ref{beral}), we get 
\[
k_t(x, y) = \int_M k^{\#}(x, z)p(t, z, y)dz = e^{tL_y}k^{\#}(x, y).
\]

Interchanging $x$ and $y$ and using the symmetry of the integral kernels, we get
\[
k_t(x, y) = k_t(y, x) = e^{tL_x}k^{\#}(y, x) = e^{tL_x}k^{\#}(x, y).
\]

Henceforth, we shall drop the subscript $x$ in $L_x$ and $L$ will refer to a differential operator in the $x$-variable, unless otherwise mentioned explicitly. To show (\ref{kilikili}), %observing that $k^{\#}(x, y)$ is supported within a distance 1 from the boundary, 
all we want is a uniform bound on 
\beq\label{ub}
\V e^{tL}k^{\#}(., y) - k^{\#}(., y)\V_{L^1(M \setminus B_{\sqrt{t}}(y))}.
\eeq

To derive (\ref{ub}), it is clear (by the Mean value theorem) that a uniform bound on 
\[
\V e^{t'L}tLk^{\#}(., y)\V_{L^1(M \setminus B_{\sqrt{t}}(y))},
\]
where $t' \in (0, t]$, will suffice. 
%Grigoryan big red book Theorem 4.9 (iv) pp 115
Now, since $k^{\#}(x, y)$ is a fixed kernel, we can choose $t$ small enough such that 
\[
\V e^{t'L}tLk^{\#}(., y)\V_{L^1(M \setminus B_{\sqrt{t}}(y))} \leq C\V tLk^{\#}(., y)\V_{L^1(M \setminus B_{\sqrt{t}}(y))},
\]
%Use Hille-Yosida theorem (see Appendix 1) and $\V \V_{L^1} \lesssim \V \V_{L^2}$ k^{\#} is both in L^1 and L^2 (calculation checked in black copy)
where $C$ does not depend on $t$. This latter quantity, using the relation between $k^{\#}(x, y)$ and $B_k(s, x, y)$ given by (\ref{hihi}), is equal to
\beq\label{seshnei}
\V tL\int_0^1 \psi_k (s) B_k(s, ., y)ds\V_{L^1(M \setminus B_{\sqrt{t}}(y))}.
\eeq

Now, when $s' \leq \sqrt{t}$, we have by (\ref{seshmuhurto}) that $\int_0^{s'} \psi_k(s) B(s, x, y) ds$ is supported on $\{(x, y) \in \overline{M} \times \overline{M} : \text{dist}(x, y) \leq s'\}$. So, (\ref{seshnei}) gives via (\ref{boddopr}) that,
\[
\V tL\int_0^1 \psi_k (s) B_k(s, ., y)ds\V_{L^1(M \setminus B_{\sqrt{t}}(y))} \lesssim t\int_{\sqrt{t}}^{1}\frac{1}{s} \V LB_k(s, ., y)\V_{L^1(M \setminus B_{\sqrt{t}}(y))}ds,
\]
which we must prove to be uniformly bounded. If we can prove
\[
\V LB_k(s, ., y)\V_{L^1(M \setminus B_{\sqrt{t}}(y))} \lesssim s^{-2},
\]
then we are done. Observe that this will be implied by 
\beq\label{impby}
\V LB_k(s, ., y)\V_{L^2(M \setminus B_{\sqrt{t}}(y))} \lesssim s^{-n/2 - 2}.
\eeq

This is because, from (\ref{seshmuhurto}), we see that $B_k(s, ., y)$ is supported on the ball $B_s(y) \subset \overline{M}$, so 
\begin{align*}
\V LB_k(s, ., y)\V_{L^1(M \setminus B_{\sqrt{t}}(y))}  & \lesssim |B_s(y)|^{1/2}\V LB_k(s, ., y)\V_{L^2(M \setminus B_{\sqrt{t}}(y))}\\
& \lesssim |B_s(y)|^{1/2}s^{-n/2 - 2} \lesssim s^{-2}.
\end{align*}

So, we are done if we can prove that,
\begin{equation}\label{PIA}
\V LB_k(s, ., y)\V_{L^2} \lesssim s^{-n/2 - 2}.
\end{equation}

We observe that (\ref{PIA}) is another variant of Lemma 2.2 of ~\cite{T2} and proceeds along absolutely similar lines. See the lemmas below, which finish the proof.
%Now, as $t \longrightarrow 0$, let us define $\Phi(t, x, y) = e^{tL}k(x, y) - k(x, y) \longrightarrow 0$, which will give an upper bound on $Q$ (here $L$ is thought of as an operator in the $y$-variable). If $\Phi(t, x, y) = e^{tL}k(x, y) - k(x, y)$ is not bounded, then there exists a sequence $(t_n, x_n, y_n)$ such that $\Phi(t_n, x_n, y_n) \longrightarrow \infty$. Choose a subsequence, still called $(t_n, x_n, y_n)$, such that $(t_n, x_n, y_n) \longrightarrow (0, x^*, y^*)$. Now, by a similar limiting argument we can contradict the assertion that $\Phi(t_n, x_n, y_n) \longrightarrow \infty$.%That means, for each $x \in \overline{M}$, we can choose $\epsilon(x)$ to be the biggest such number such that $\sup_{t \in (0, \epsilon(x)]} \Phi(t, x) < 1$. Though apriori this $\epsilon(x)$ actually depends on $x$, it can be shown that $\epsilon(x)$ is actually a bounded function of $x$ (using the fact that $\overline{M}$ is compact). Here is a sketch of the argument: consider a sequence of points $x_n$ such that $\epsilon(x_n) \to 0$. Then we get a convergent subsequence $\to x^*$, say. Using the fact that $\Phi(t, x_n) \to \Phi(t, x^*)$, we work out a contradiction. Putting everything together, we have our result.

\end{proof}
\begin{lemma}
\begin{equation}\label{LsL}
\V Le^{-s\sqrt{-L}}\V_{\mathcal{L}(L^2, L^\infty)} \lesssim s^{-n/2 - 2}, s \in (0, 1].
\end{equation}
\end{lemma}
\begin{proof}
For $f \in L^2(M)$, call
\begin{equation}
u_s(x) = e^{-s\sqrt{-L}}f(x), s > 0, x \in \overline{M}.
\end{equation}

Then $u$ is a solution of 
\begin{align}
(\partial^2_s + L) u = 0, & \text{   on  } (0, \infty) \times M\\
B(x, \partial_x)u = 0, & \text{   on  } (0, \infty) \times \partial M,
\end{align}
where $B$ represents the coercive boundary condition defining $\mathcal{D}(L)$. We have, by the Hille Yosida theorem,
\[
\V u_s\V_{L^2(M)} = \V e^{-s\sqrt{-L}}f\V_{L^2} \leq \V f\V_{L^2(M)}, \forall s > 0.
\]

Let us pick $\delta \in (0, 1], s_0 = \delta,$ and $x_0 \in \overline{M}$. Let $U = \{x \in \overline{M} : \text{dist} (x, x_0) < 2\delta\}$. We now scale the $s$ and the $x$ variables by a factor of $1/\delta$, and let $v_s(x)$ denote the new function corresponding to $u_s$ in the scaled variables. Then $v$ solves
\begin{align}
(\partial^2_s + \tilde{L}) v = 0, & \text{   on  } (1/2, 3/2) \times \tilde{U},\\
\tilde{B}(x, \partial_x)v = 0, & \text{   on  } (1/2, 3/2) \times (\tilde{U} \cap \partial M),
\end{align}
which is a coercive boundary valued elliptic system with uniformly smooth coefficients and uniform ellipticity bounds. On calculation,
\[
\V v\V_{L^2((1/2, 3/2) \times \tilde{U})} \approx \delta^{-n/2}\V u\V_{L^2((\delta/2, 3\delta/2) \times U)} \lesssim \delta^{-n/2}\V f\V_{L^2}.\]

From elliptic regularity estimates %(see ~\cite{GT}, Chapter 6), 
we get %$C^{2, \alpha}$ bounds on $v$ which imply, in particular, 
that 
\begin{align}\label{sesh}
\V Lv_1(.)\V_{L^\infty(\tilde{U_0})} & \lesssim \V v\V_{L^2((1/2, 3/2) \times \tilde{U})} \lesssim \delta^{-n/2}\V f\V_{L^2}.
\end{align}
This can be obtained by iterating the estimate (11.29) of Chapter 5 of \cite{T6} to prove that $\V u\V_{H^k(I \times \tilde{U_0} )} \lesssim \V u\V_{L^2((1/2, 3/2) \times \tilde{U})}$, where $1 \in I \subset (1/2, 3/2)$, and taking $k$ high enough such that $H^k \hookrightarrow C^{2, \alpha}$ for some $\alpha$. This implies (\ref{sesh}). Scaling back gives our result %(here we finally get to use our assumption (\ref{ASS})),
\[
|Lu_{s_0}(x_0)| \lesssim \delta^{-n/2 - 2}\V f\V_{L^2}.\]
\end{proof}
\begin{lemma}
If $G : \mathbb{R} \rightarrow \mathbb{R}$ satisfies 
\[
|G(\lambda)| \lesssim (1 + |\lambda|)^{-\gamma - 1}, \gamma > \frac{n}{2},\]

then
\begin{equation}\label{Sting}
\V LG(s\sqrt{-L})\V_{\mathcal{L}(L^2, L^\infty)} \lesssim s^{-n/2 -2}, s \in (0, 1].
\end{equation}

This implies (\ref{PIA}).
\end{lemma}
\begin{proof}
We start by using the formula
\[
(I + s\sqrt{-L})^{-\sigma} = \frac{1}{\Gamma(\sigma)}\int_0^\infty e^{-t}e^{-ts\sqrt{-L}}t^{\sigma - 1}dt.
\]

That implies, in conjunction with (\ref{LsL}), 
\begin{align}\label{hogan}
\V L(I + s\sqrt{-L})^{-\sigma}\V_{\mathcal{L}(L^2, L^\infty)} & \lesssim \int^{1/s}_0 e^{-t}(st)^{-n/2 - 2}t^{\sigma - 1}dt \nonumber \\ & + \int^{\infty}_{1/s} e^{-t}\V Le^{-ts\sqrt{-L}}\V_{\Cal{L}(L^2, L^\infty)}t^{\sigma - 1}dt \nonumber \\ 
& \lesssim s^{-n/2 - 2} + 1,
\end{align}
where $\sigma > n/2 + 2$. Also, in the above calculation, we have used that when $r \geq 1$,
\begin{align*}
\V Le^{-r\sqrt{-L}}\V_{\Cal{L}(L^2, L^\infty)} & = \V Le^{-\sqrt{-L}}e^{-(r - 1)\sqrt{-L}}\V_{\Cal{L}(L^2, L^\infty)}\\
& \lesssim \V e^{-(r - 1)\sqrt{-L}}\V_{\Cal{L}(L^2, L^2)} \leq 1.
\end{align*}

The facts that (\ref{hogan}) implies (\ref{Sting}) and (\ref{Sting}) implies (\ref{PIA}) are absolutely similar to the proof of Lemma \ref{label}.
\end{proof}

\section{$L^\infty - \text{BMO}_L$ continuity}\label{3.6}
%Before we start the section proper, let us make a notational convention: for the rest of this section, by $T: X \to Y$, we will mean $T$ is a continuous operator between function spaces $X$ and $Y$. \newline
%To establish $L^\infty$-BMO$_L$ continuity, first we argue that if 
%\begin{align}
%\varphi(\sqrt{-L}) = \varphi^b(\sqrt{-L}) + \varphi^{\#}(\sqrt{-L}),
%\end{align}
%then it is enough to prove the result for $\varphi^{\#}(\sqrt{-L})$. This is because, we have 
%\beq
%|\varphi^b(\lambda)| \lesssim (1 + |\lambda|)^{-m} \mbox{   for some   } m > \frac{n}{2},
%\eeq
%where dim $M = n$. \newline
%With that in place, as in (\ref{be4}), we have
%\beq \label{1}
%\varphi^b(\sqrt{-L}) : L^1(M) \longrightarrow L^2(M).
%\eeq
We see that (\ref{be4}) gives, by duality, 
\beq\label{97}
\varphi^b(\sqrt{-L}) : L^2 \longrightarrow L^\infty.
\eeq
%Also, from (\ref{1.12}) we have
%\beq
%\varphi^b(\sqrt{-L}): L^2 \longrightarrow H^m(M) \subset L^1(M)
%\eeq
%which gives
%\beq\label{99}
%\varphi^b(\sqrt{-L}) : L^2 \longrightarrow L^1
%\eeq
%and by duality
%\beq\label{4}
%\varphi^b(\sqrt{-L}) : L^\infty \longrightarrow L^2
%\eeq

Now, we can prove the following inclusion on a compact manifold:
\beq\label{98}
\V .\V_{\text{BMO}_L} \lesssim \V .\V_{L^\infty}.
\eeq
%\end{lemma}
%\begin{proof}
%\begin{align*}
%\frac{1}{|B_{\sqrt{t}}(y)|}\int_{B_{\sqrt{t}}(y)} |f(x) - e^{tL}f(x)|dx & \leq \frac{1}{\sqrt{|B_{\sqrt{t}}(y)|}}\V f - e^{tL}f\V_{L^2}\\
%& \lesssim \frac{1}{\sqrt{|B_{\sqrt{t}}(y)|}}\V f\V_{L^2} \lesssim \frac{1}{\sqrt{|B_{\sqrt{t}}(y)|}}\V f\V_{L^\infty}\sqrt{|B_{\sqrt{t} + \delta}(y)|}\\
%& \lesssim \V \psi f\V_{L^\infty}
%\end{align*}
%\end{proof}

This is because, for small enough $t > 0$, we have
\[
\frac{1}{|B_{\sqrt{t}}(y)|}\int_{B_{\sqrt{t}}(y)} |f(x) - e^{tL}f(x)|dx \leq \V f - e^{tL}f\V_{L^\infty} \lesssim \V f\V_{L^\infty}.
\]

(\ref{97}) and (\ref{98}) give
\beq
\V\varphi^b(\sqrt{-L})f\V_{\text{BMO}_L} \leq \V\varphi^b(\sqrt{-L})f\V_{L^\infty} \lesssim \V f\V_{L^2},
\eeq

which means
\beq\label{5}
\varphi^b(\sqrt{-L}) : L^2 \longrightarrow \text{BMO}_L.
\eeq

(\ref{98}) and (\ref{5}) give
\beq\label{6}
\varphi^b(\sqrt{-L}) : L^\infty \longrightarrow \text{BMO}_L.
\eeq
%So now, we pick $f \in L^\infty$ and try to show $\varphi^{\#}(\sqrt{-L})f \in \mbox{BMO}_L$.
 
Finally, we have
\begin{proposition}\label{4.4}
\[
\varphi^{\#}(\sqrt{-L}) : L^\infty \longrightarrow \text{BMO}_L.
\]
\end{proposition}
\begin{proof}
\begin{flalign*}
& \frac{1}{|B_{\sqrt{t}}(y)|}\int_{B_{\sqrt{t}}(y)}|\varphi^{\#}(\sqrt{-L})f(x) - e^{tL}\varphi^{\#}(\sqrt{-L})f(x)|dx \\
& \leq \frac{1}{|B_{\sqrt{t}}(y)|}\int_{B_{\sqrt{t}}(y)}|\varphi^{\#}(\sqrt{-L})\psi(x)f(x) - e^{tL}\varphi^{\#}(\sqrt{-L})\psi(x)f(x)|dx \\
& + \frac{1}{|B_{\sqrt{t}}(y)|}\int_{B_{\sqrt{t}}(y)}|\varphi^{\#}(\sqrt{-L})(1 - \psi(x))f(x) - e^{tL}\varphi^{\#}(\sqrt{-L})(1 - \psi(x))f(x)|dx \\
& = \Psi_1 + \Psi_2.
\end{flalign*}
where $\psi$ is a smooth cut-off function supported in $B_{\sqrt{t} + \delta}(y)$, and $\psi(x) \equiv 1$ on $B_{\sqrt{t}}(y)$. Clearly, by H{\"o}lder's inequality,
\begin{align*}
\Psi_1 & = \frac{1}{|B_{\sqrt{t}}(y)|}\int_{B_{\sqrt{t}}(y)} |\varphi^{\#}(\sqrt{-L})\psi(x)f(x) - e^{tL}\varphi^{\#}(\sqrt{-L})\psi(x)f(x)|dx\\
& \leq \frac{1}{\sqrt{|B_{\sqrt{t}}(y)|}}\V \varphi^{\#}(\sqrt{-L})\psi f - e^{tL}\varphi^{\#}(\sqrt{-L})\psi f\V_{L^2}\\
& \lesssim \frac{1}{\sqrt{|B_{\sqrt{t}}(y)|}}\V \varphi^{\#}(\sqrt{-L})\psi f\V_{L^2} \text{  (by contractivity of heat semigroup)}\\
& \lesssim \frac{1}{\sqrt{|B_{\sqrt{t}}(y)|}}\V \psi f\V_{L^2} \leq \frac{1}{\sqrt{|B_{\sqrt{t}}(y)|}}\V \psi f\V_{L^\infty}\sqrt{|B_{\sqrt{t} + \delta}(y)|} \lesssim \V \psi f\V_{L^\infty} \leq \V f\V_{L^\infty}.
\end{align*}

Also
\begin{align}
|\varphi^{\#}(\sqrt{-L})(I - e^{tL})(1 - \psi(x))f(x)| & \leq \int_{M \setminus B_{\sqrt{t}}(x)}|(k^{\#}(x, z) - k_t(x, z))(1 - \psi(z))f(z)|dz\\
& \leq \V f\V_{L^\infty}\int_{M \setminus B_{\sqrt{t}}(x)}|k^{\#}(x, z) - k_t(x, z)|dz 
\end{align}
where $k_t$ is the integral kernel of $\varphi^{\#}(\sqrt{-L})e^{tL}$. Now, choosing $t$ small\footnote{This is the main reason for introducing $BMO^\epsilon_L$ instead of just the usual $BMO_L$}, % such that $t < \epsilon$, where $\frac{1}{2\sqrt{\epsilon}} \geq \text{diam } M$, 
we are done by Lemma \ref{kl} and Lemma \ref{Indep}.
%\begin{align}
%\int_{M \setminus B_{\sqrt{t}}(x)}|k^{\#}(x, z) - k_t(x, z)|dz & = \int_{M \setminus B_{\sqrt{t}}(x)}|k^{\#}(z, x) - \int_M k^{\#}(z, w)h(t, w, x)dw|dz\\
%& = \int_{M \setminus B_{\sqrt{t}}(x)}|k^{\#}(z, x) - \int_M k^{\#}(z, w)p(t, w, x)dw|dz
%\end{align}
%where $h$ is the heat-kernel. We have already proved that the last quantity is bounded above. That proves 
%\beq
%\varphi^{\#}(\sqrt{-L}) : L^\infty \longrightarrow \mbox{BMO}_L
%\eeq
\end{proof}
%\begin{remark}
%By tweaking the above calculations, we can bring out the BMO-norm instead of the $L^\infty$ norm in the first equation. And in the second, we can write $L^2$ norm of kernel difference times $L^2$ norm of $f$, the latter begin less than BMO-norm. The proof that controls $L^1$ norm of kernel difference will still control the $L^2$ norm of kernel difference for small $t$.
%\end{remark}
So, let us see the immediate consequence of Proposition \ref{4.4}. By virtue of this, we immediately have $L^p$-continuity of $\varphi(\sqrt{-L})$ upon application of Theorem 5.6 of ~\cite{DY1}. This $L^p$-continuity result is not new of course. It is established in a more general context in ~\cite{DOS}. ~\cite{T2} has a different proof of this same result. %However, we believe that our proof is rather different from both of these. 
%Also, it must be remarked that Lemma \ref{kl}, in conjunction with Theorem 1 of ~\cite{DM} also implies directly that $\varphi(\sqrt{-L})$ is of weak type (1,1), which by application of the Marcinkiewicz interpolation techniques, will give another proof of the $L^p$-continuity (also refer to Lemma B.1 of ~\cite{T6} for a discussion of Theorem 1 of ~\cite{DM}).

Note also the $L^\infty-\text{BMO}_L$ result in Theorem 6.2 of ~\cite{DY1}. However, the conditions used to prove Theorem 6.2 in ~\cite{DY1} are stronger than (\ref{g}). So, in comparison, it can be said that we prove similar results in a more restricted setting, but with less assumptions on the heat semigroup $e^{tL}$. 

It is also a natural question to ask what happens if we adopt the seemingly more natural definition of $\text{BMO}$ spaces, as follows:
\beq
\V f\V_{\text{BMO}^\epsilon} = \sup_{B \in \Cal{B}} \frac{1}{|B|} \int_B |f(x) - A_tf(x)|dx,
\eeq
where $\Cal{B}$ contains all balls of radius less than or equal to $\epsilon$, $B$ is a ball of radius $\sqrt{t}$, and $A_t$ is the operator whose integral kernel is given by 
\[
h(t, x, y) = \frac{1}{|B_{\sqrt{t}}(y)|}\chi_{B_{\sqrt{t}}(y)}(x).
\]

On calculation, it can be seen that estimates on 
\[
\int_{d(x, y)\geq \sqrt{t}, d(y, z)\geq \frac{\sqrt{t}}{2}}|k^{\#}(x, y) - k^{\#}(x, z)|dx\]
will imply that 
\beq
\sup_{y \in \overline{M}}\sup_{t \in (0, \varepsilon]}\int_{d(x, y) \geq \sqrt{t}} |k^{\#}(x, y) - k_t(x, y)|dx \leq C,
\eeq
where $k_t$ represents the integral kernel of $\varphi^{\#}(\sqrt{-L})A_t$. The issue is, here $A_t$ and $\varphi^{\#}(\sqrt{-L})$ do not necessarily commute.
%\subsection{$\text{BMO}_L - \text{BMO}_L$ boundedness}
%(\ref{5}) and (\ref{6}) interpolate to give that 
%\beq
%\varphi^b(\sqrt{-L}) : \text{BMO}_L \longrightarrow \text{BMO}_L
%\eeq
%Also, following the proof of Proposition (\ref{4.4}), we can see that
%\beq
%A \lesssim ||f||_{\text{BMO}_L}
%\eeq
%Also, we observe that
%\begin{align}
%|\varphi^{\#}(\sqrt{-L})(I - e^{tL})(1 - \psi(x))f(x)| & \leq \int_{M \setminus B_{\sqrt{t}}(x)}|(k^{\#}(x, z) - k_t(x, z))(1 - \psi(z))f(z)|dz\\
%& \leq ||f||_{L^2}||k^{\#}(x, z) - k_t(x, z)||_{L^2(M \setminus B_{\sqrt{t}}(x))} 
%\end{align}
%So, if we can bound
%\[
%||k^{\#}(x, z) - k_t(x, z)||_{L^2(M \setminus B_{\sqrt{t}}(x))} 
%\]
%at least for small $t$, we would be done. It is clear on observation that this follows from Key Lemma (\ref{kl}).
%\newline
%\\
%All the above give us the following continuity:
%\begin{theorem}
%\beq
%\varphi(\sqrt{-L}) : \text{BMO}_L \longrightarrow \text{BMO}_L
%\eeq
%\end{theorem}
\section{Appendix 1: properties of heat and Poisson semigroups}\label{GFWK}

In this Appendix, we include some essential facts about the semigroups $e^{tL}$ and $e^{-t\sqrt{-L}}$. Henceforth, we will call them heat and Poisson semigroups respectively. Since $-L$ is a nonnegative semi-definite self-adjoint operator, by the Hille-Yosida theorem (see ~\cite{Sc}, Proposition 6.14), $e^{tL}$ gives a contraction semigroup on $L^2(M)$. Our first lemma is the following %(see Notation \ref{LIP} of Appendix D) 
\begin{lemma}\label{3.2}
	\beq
	\V e^{tL}\V_{\Cal{L}(L^2, \text{Lip})} \lesssim (1 + t^{-n/4 - 1/2}),\text{   } t>0.
	\eeq
\end{lemma}
\begin{proof}
	
	We will first use the gradient estimate (\ref{g1}) to prove that
	\beq \label{SIMILA}
	\int_{M} |\nabla_x p(t, x, y)|^2 dy \lesssim t^{-n/2 - 1}, t \in (0, 1].
	\eeq
	
	Using (\ref{g1}), we see that \[
	\int_{M} |\nabla_x p(t, x, y)|^2 dy \lesssim t^{-n - 1}\int_{M} e^{-2\kappa d(x, y)^2/t}dy.\]
%	\eeq

	%If we can prove $\int_\overline{M} e^{-\kappa d(x, y)^2/t}dy \lesssim t^{n/2}$, we are done by the compactness of $M$. \newline
	Now we consider the identity mapping $i : (\overline{M}, g) \longrightarrow (\overline{M}, \frac{t}{2\kappa}g)$, where $(\overline{M}, \frac{t}{2\kappa}g)$ denotes the manifold $\overline{M}$ with a scaled metric. That gives, 
	\begin{align*}
	\int_{M} e^{-2\kappa d(x, y)^2/t} dy & = \int_{M} e^{-d(x, z)^2} |Ji|dz \approx  t^{n/2}\int_{M} e^{-2d(x, z)^2} dz  \approx t^{n/2},
	\end{align*}
	where $Ji$ denotes the Jacobian of the map $i$, %and we also use the compactness of $\overline{M}$, 
	which finally gives (\ref{SIMILA}). 
	
	We have, as usual,
	%\[
	%e^{tL}f(x) = \int_M p(t, x, y)f(y) dy,
	%\]
	%implying (by differentiating under the integration sign)
	\begin{align*}
	\nabla_x e^{tL}f(x) & = \int_M \nabla_x p(t, x, y) f(y) dy \leq \V\nabla_x p (t, x, .)\V_{L^2}\V f\V_{L^2},
	\end{align*}
	which gives, by (\ref{SIMILA}),
	\beq \label{esti1}
	\V e^{tL}\V_{\Cal{L}(L^2,\text{Lip})} \leq t^{-n/4 - 1/2}, t \in (0, 1].
	\eeq
	
	Now, if 
	$
	\text{Spec}(-L) \subset [\rho, \infty),
	$
	%then for $t > 0$ we have 
	%\beq \label{esti2}
	%|\nabla e^{tL} f(x)| \lesssim e^{-t\rho}||f||_{L^2}
	%\eeq
	%This is because, 
	then 
	$
	\V e^{tL} f\V_{L^2} \leq e^{-t\rho}\V f\V_{L^2},
	$
	which, in conjunction with  
	%\[
	%\V e^{tL} f\V_{L^2} \leq e^{-(t - 1)\rho}\V e^{L}f\V_{L^2} \]
	(\ref{esti1}) means that for $t > 1$,
	\begin{align*}
	|\nabla e^{tL} f(x)| & = |\nabla e^{L/2}e^{L/2}e^{(t - 1)L}f(x)|\\
	& \lesssim \V e^{L/2}e^{(t - 1)L}f\V_{L^2} \text{   from  } (\ref{esti1})\\
	& \leq \V  e^{(t - 1)L}f\V_{L^2} \text{   (contractivity of heat semigroup)}\\
	& \leq e^{-\rho(t - 1)}\V f\V_{L^2} \lesssim \V f\V_{L^2}.
	\end{align*}
	
	So, putting (\ref{esti1}) and the last inequality together, %, and by the compactness of $\overline{M}$, 
	we have 
	\[
	\V e^{tL}\V_{\Cal{L}(L^2, \text{Lip})} \lesssim (1 + t^{-n/4 - 1/2}),\text{   } t > 0.
	\]
	
\end{proof}
%Also, we use the above estimates to derive the following conclusions about the Poisson semigroup.

Similarly, for the Poisson semigroup, we have
\begin{lemma}
	\beq\label{hghghg}
	\V e^{-t\sqrt{-L}}\V_{\Cal{L}(L^2, L^\infty)} \lesssim (1 + t^{-n/2}),\text{  } t > 0,
	\eeq
	
	and
	\beq
	\V e^{-t\sqrt{-L}}\V_{\Cal{L}(L^2, \text{Lip})} \lesssim (1 + t^{-n/2 - 1}), \text{  }t > 0.
	\eeq
\end{lemma}
\begin{proof}
	As in Lemma \ref{3.2}, starting from (\ref{g}), we can establish that 
	\beq
	\V e^{tL}\V_{\Cal{L}(L^2, L^\infty)} \lesssim (t^{-n/4} + 1), \text{  }t > 0.
	\eeq
	
	Now the estimate on $e^{-t\sqrt{-L}}$ can be obtained from the Subordination identity (see (5.22), Chapter 3 of \cite{T6}),
	\begin{align}
	e^{-t\sqrt{-L}} & \approx \int^\infty_0 te^{-t^2/4s}s^{-3/2}e^{sL}ds.
	\end{align}
	
	This gives,
	\begin{align}
	\V e^{-t\sqrt{-L}}\V_{\Cal{L}(L^2, L^\infty)} & \lesssim \int^\infty_0 te^{-t^2/4s}s^{-3/2}\V e^{sL}\V_{\Cal{L}(L^2, L^\infty)}ds \lesssim \int^\infty_0 te^{-t^2/4s}s^{-3/2}(s^{-n/4} + 1)ds\\
	& \lesssim t\int_0^\infty e^{-t^2/4s}s^{-\frac{6 + n}{4}}ds + t\int_0^\infty e^{-t^2/4s}s^{-3/2}ds\nonumber.
	\end{align}
	
	Calling the first integral above $I_1$ and the second one $I_2$, we get $	I_2  \leq c_0,$ where $c_n$ is a multiple of $\Gamma(\frac{n + 1}{2})$ (see ~\cite{T6} for details, particularly pp 247-248).
	Similarly, 
	\begin{align*}
	I_1 & = t\int_0^\infty e^{-t^2/4s}s^{-\frac{3 + n/2}{2}}ds \leq c_{n/2}\frac{t}{(t^2)^{\frac{n + 2}{4}}}
\leq c_{n/2}t^{-n/2}.
	\end{align*}
	
	This gives (\ref{hghghg}).
	%This gives us an estimate on $\V e^{tL}\V_{\Cal{L}(L^2, L^\infty)}$.
	Now, when $t \in (0, 1]$, we can write,
	\begin{align*}
	\V\nabla e^{-t\sqrt{-L}} f\V_{L^\infty} & = \V\nabla e^{tL}e^{-t\sqrt{-L}}e^{- tL}f\V_{L^\infty} \lesssim (t^{-n/4 - 1/2})\V e^{-t\sqrt{-L}}e^{-tL}f\V_{L^2} \text{   from  } (\ref{esti1})\\
	& \lesssim (t^{-n/4 - 1/2})\V e^{-tL}f\V_{L^2} \lesssim (t^{-n/4 - 1/2})\V f\V_{L^2}.
	\end{align*}
	
	Lastly, when $t \in [1, \infty)$,
	\begin{align*}
	\V\nabla e^{-t\sqrt{-L}} f\V_{L^\infty} & = \V\nabla e^{-\sqrt{-L}}e^{-(t - 1)\sqrt{-L}}f\V_{L^\infty} \lesssim \V e^{-(t - 1)\sqrt{-L}}f\V_{L^2} \lesssim \V f\V_{L^2}.
	\end{align*}
	This proves the lemma.
\end{proof}

\section{Appendix 2: finite propagation speed of $\text{cos }t\sqrt{-L}$.}\label{FPSS}
In this section we investigate some sufficient criteria for $\text{cos }t\sqrt{-L}$ to have finite propagation speed under special boundary conditions. Namely, we will establish finite speed of propagation for those $L$ which can be written as 
\beq\label{special}
-L = D^*D + H
\eeq
where $D$ is a first order elliptic differential operator with either the generalized Dirichlet or Neumann boundary condition (see (\ref{DC}) and (\ref{NC}) below), and $H \in L^2(M)$ is nonnegative. %Also, for this section, we will go a bit more general and work out finite propagation speed for $\text{cos }t\sqrt{-L}$, where $L$ is actually an elliptic system, acting on sections of vector bundles. %For simplicity, we make use of the so-called Davies-Gaffney estimates on the operator $\mbox{cos }t\sqrt{-L}$ for $-L = D^*D + H$, where $D$ is a first order elliptic operator with either the Dirichlet or Neumann boundary condition. 
To do this, we invoke the so-called Davies-Gaffney estimates: % (see Definition \ref{DGEL} of Appendix B). 
%to recall, 
\begin{definition} 
	An operator $L$ satisfies the Davies-Gaffney estimates on a manifold $M$ if
	\beq\label{Davies}
	(e^{tL}u, v) \leq e^{-\frac{r^2}{4t}}\V u\V_{L^2}\V v\V_{L^2}
	\eeq
	for all $t > 0$, for all pairs of open subsets $U, V$ of $M$, $\text{supp } u \subset U$, $\text{supp }v \subset V$, sections $u \in L^2(U), v \in L^2(V)$ and $r = \text{dist }(U, V)$, the metric distance between $U$ and $V$.
\end{definition}
We also recall the following (see ~\cite{S}, Theorem 2).
\begin{lemma}\label{sik}
	For a self-adjoint negative semi-definite operator $L$ on $L^2(M)$, satisfaction of the Davies-Gaffney estimates is equivalent to finite propagation speed property of $\text{cos }t\sqrt{-L}$. Furthermore, it is enough to check the Davies-Gaffney estimates only for open sets $U, V$ which are balls around some points.
\end{lemma}
It might be pointed out that this method is eminently suited to establishing finite propagation speed type results particularly when the manifold has boundary or is less ``nice'' in some other way, as Lemma \ref{sik} above holds in the great generality of metric measure spaces $(X, d, \mu)$, where $\mu$ is a Borel measure with respect to the topology defined by $d$.

%So, to establish the finite propagation speed of $\mbox{cos }t\sqrt{-L}$ in our current setting, we attempt to establish the Davies-Gaffney estimates. L
Now, let $-L = D^*D + H$, where $D : H^{1}(\overline{M}, E) \rightarrow H^2(\overline{M}, F)$ be a first-order differential operator %\footnote{Interpreted as operators factoring through first order jet bundles. Note that by the Peetre theorem, differential operators are also exactly the linear operators which decrease support, i.e., have the local property.} 
between sections of vector bundles. Assume that the symbol $\sigma_D(x, \xi) : E_x \rightarrow F_x$ is injective for $x \in \overline{M}, \xi \in T^*_x\overline{
	M}\setminus \{0\}$. Following ~\cite{T2}, consider the following generalization of the Dirichlet condition on $\Cal{D}(D)$:
\beq\label{DC}
u \in \mathcal{D}(D) \Rightarrow \beta(x)u(x) = 0, \text{    }\forall x \in \pa M,
\eeq
where $\beta(x)$ is an orthogonal projection on $E_x$ for all $x \in \pa M$. We also consider the following generalization of the Neumann boundary condition:
\beq \label{NC}
u \in \mathcal{D}(D) \Rightarrow \gamma(x)\sigma_D (x, \nu) u(x) = 0 , \text{    }\forall x \in \pa M,
\eeq
where $\nu(x)$ is the outward unit normal to $\pa M$ and $\gamma(x)$ is an orthogonal projection on $E_x$ for all $x \in \pa M$.

We first argue that both these boundary conditions have the consequence that
\beq\label{conclu}
\langle \sigma_D(x, \nu) v, w\rangle = 0, \text{    }\forall x \in \pa M,
\eeq
when $v \in \Cal{D}(D), w \in \Cal{D}(D^*)$ and $v, w$ are smooth. 
%write about the Neumann condition
 This is because
\[
\int_M(\langle Dv, w\rangle - \langle v, D^*w\rangle) dV = \frac{1}{i}\int_{\partial M}\langle \sigma_D(x, \nu)v, w\rangle dS.
\]

Now, $w \in \Cal{D}(D^*)$ implies that the left hand side in the above equation vanishes. So, for the Dirichlet boundary condition, for smooth $v, w$ we have 
\begin{equation}\label{DBC}
w \in \mathcal{D}(D^*) \Longrightarrow (I - \beta(x))\sigma_D(x, \nu)^*w(x) = 0, x \in \partial M,
\end{equation}
where $\nu$ is the outward unit normal to $\partial M$. This gives, for smooth $v$ and $w$, 
\[
v \in \mathcal{D}(D), w \in \mathcal{D}(D^*) \Longrightarrow \langle \sigma_D(x, \nu)v, w\rangle = 0 \text{   on   } \partial M.
\]

For the Neumann boundary condition, (\ref{DBC}) will be replaced by 
\begin{equation}\label{NBC}
w \in \mathcal{D}(D^*) \Longrightarrow (I - \gamma(x))w(x) = 0, x \in \partial M
\end{equation}
with the same conclusion (\ref{conclu}). With that in place, we can now prove Proposition \ref{FSOP}.

\begin{proof}
	We first observe that by the Cauchy-Schwarz inequality
	\beq
	(e^{tL}u, v) \leq \V\chi_V e^{tL}u\V_{L^2}\V v\V_{L^2},
	\eeq
	where $\chi_V$ represents the characteristic function of $V$. So, to get finite speed of propagation, we want to establish (\ref{Davies}), 
	which will in turn be implied by % , we are looking to establish a relation like
	\beq
	\V\chi_V e^{tL}u\V_{L^2} \leq e^{-\frac{r^2}{4t}}\V u\V_{L^2},
	\eeq
	when supp $ u \subset U$. Now, let $w = \chi_V e^{tL}u$ and call $\rho = \frac{r}{t}$. Then we have
	\begin{align}\label{thenwehave}
	\int_V |w|^2 dx & \leq e^{-\rho r}\int_V \langle w, w\rangle e^{\varphi(x)}dx \leq e^{-\rho r}\int_M \langle e^{tL}u, e^{tL}u\rangle e^{\varphi(x)}dx.
	\end{align}
	 
	Let us define
	\beq\label{ETI}
	E(t) = \int_M \langle e^{tL}u, e^{tL}u\rangle e^{\varphi(x)}dx.
	\eeq
	
	Differentiating (\ref{ETI}) with respect to $t$, we get
	\begin{align*}
	\frac{1}{2}E'(t) & = \text{Re}\int_M \langle\partial_t e^{tL}u, e^{tL}u\rangle e^{\varphi(x)}dx = \text{Re}\int_M \langle Le^{tL}u, e^{tL}u\rangle e^{\varphi(x)}dx \\
	& = - \text{Re}\int_M \langle (D^*D + H) e^{tL}u, e^{tL}u\rangle e^{\varphi(x)}dx\\
	& = - \text{Re}\int_M \langle De^{tL}u, D(e^{tL}ue^{\varphi(x)})\rangle dx - \text{Re}\int_M H\langle e^{tL}u, e^{tL}u\rangle e^{\varphi(x)}dx \\
	& + \text{Re}\frac{1}{i} \int_{\pa M}\langle \sigma (x, \nu)e^{tL}u e^{\varphi(x)}, De^{tL}u\rangle dS\\
	& = - \text{Re} \int_M (\langle De^{tL}u, De^{tL}u\rangle e^{\varphi(x)} + \langle De^{tL}u, [D, e^{\varphi(x)}]e^{tL}u\rangle) dx \\ 
	& - \text{Re}\int_M H\langle e^{tL}u, e^{tL}u\rangle e^{\varphi(x)}dx + \mbox{     }\text{Re}\frac{1}{i} \int_{\pa M}\langle \sigma (x, \nu)e^{tL}u e^{\varphi(x)}, De^{tL}u\rangle dS\\
	& \leq - \text{Re} \int_M (\langle De^{tL}u, De^{tL}u\rangle e^{\varphi(x)} + \langle De^{tL}u, [D, e^{\varphi(x)}]e^{tL}u\rangle) dx,
	\end{align*}
	using the facts that $H \geq 0$ and that under the Dirichlet or the Neumann boundary condition, the last term $\int_{\pa M}\langle \sigma_D (x, \nu)e^{tL}u e^{\varphi(x)}, Du\rangle dS$ disappears. Now, if we can say that %the quantity $\frac{1}{2}E'(t)$ is 
	\begin{align*}
	\frac{1}{2}E'(t) \leq \frac{\rho^2}{4}\int_M \langle e^{tL}u, e^{tL}u\rangle, e^{\varphi(x)}dx,
	\end{align*}
	then we will be in a position to use Gronwall's inequality.
	
	Now what is the condition that allows this? Let us define $P = [D, e^{\varphi }]$. Now we have
	\begin{align*}
	4(De^{tL}u, Pe^{tL}u) & = 4(e^{\varphi /2}De^{tL}u, e^{-\varphi /2}Pe^{tL}u) \leq 4\V e^{\varphi /2}De^{tL}u\V^2_{L^2} + \V e^{-\varphi /2}Pe^{tL}u\V^2_{L^2}.
	\end{align*}
	
	So it seems that the correct condition is to demand that 
	\[
	\V e^{-\varphi /2}Pe^{tL}u\V^2_{L^2} \leq \rho^2\V e^{\varphi /2}e^{tL}u\V^2_{L^2}
	\]
	or, 
	\beq 
	\V e^{-\varphi/2}Pv\V_{L^2} \leq \rho\V e^{\varphi/2}v\V_{L^2}.
	\eeq
	%This is equivalent to claiming that
	%\beq
	%e^{-\varphi}P : e^{\varphi/2}L^2 \longrightarrow e^{\varphi/2}L^2 
	%\eeq
	%where $e^{\varphi/2}L^2$ is a Hilbert space. 
	
	Heuristically, we can say that a condition like this is expected, as the propagation phenomenon of $\text{cos }t\sqrt{-L}$ will be dictated by the interaction of $L$, and hence of $D$ with the distance function on $\overline{M}$.
	%To put this in a more familiar form we observe that trivially
	%\[
	%||e^{-\varphi(x)/2}Pe^{tL}u||^2_{L^2}  \leq \int\langle Pe^{tL}u, Pe^{tL}u\rangle
	%\]
	%and 
	%\[
	%||e^{-\varphi(x)/2}e^{tL}u||^2_{L^2}  \leq \int\langle e^{tL}u, e^{tL}u\rangle e^{\varphi(x)}
	%\]
	%So we finally name our condition as 
	%\beq \label{cond}
	%||Pe^{tL}u||_{L^2} \leq \rho ||e^{tL}||_{L^2}
	%\eeq
	%Now, 
	%\[
	%P = [D, e^{\varphi(x)}] \in OPS^0_{1, 0} : L^2 \longrightarrow L^2
	%\]
	%where $OPS^0_{1, 0}$ denotes the usual pseudodifferential class. This gives
	%\begin{align*}
	%4(De^{tL}u, Pe^{tL}u) & \leq 4||De^{tL}u||_{L^2}^2 + ||Pe^{tL}u||_{L^2}^2\\
	%& \lesssim 4||De^{tL}u||_{L^2}^2 + ||e^{tL}u||_{L^2}^2
	%\end{align*}
	%\footnote{expand this explanation; use the scribbled note to explain more about the Neumann boundary condition. Also recall the precise form of the ``Dirichlet boundary condition'' from lpmb}.
	 
	So, now we can say
	\beq 
	E'(t) \leq \rho^2/2 E(t).
	\eeq
	
	This gives, by Gronwall's inequality, $E(t) \leq e^{\rho^2t/2} E(0)$. 
	%So at last we have
	%\begin{align}
	%E(t) \leq e^{\alpha t}E(0) & = e^{\alpha t}\int_M (u, u)e^{\varphi (x)}dx \\
	%& = e^{\alpha t} ||u||^2_{L^2}
	%\end{align}
	Plugging everything back, we have from (\ref{thenwehave}),
	\[
	\int_V |w|^2 dx \leq e^{\rho^2t/2 - \rho r}\V u\V^2_{L^2}.
	\]
	
	Using $\rho = r/t$, we have
	\beq
	\int_V |w|^2 dx \leq e^{-\frac{r^2}{2t}}\V u\V^2_{L^2}.
	\eeq
	This proves what we want. %, modulo the fact that we need to extend the result to include all functions $H$ such that $-L - H$ is positive definite. For this, see below.
\end{proof}

\begin{remark}
	Though (\ref{cond}) does not seem to be much of an improvement over (\ref{Davies}), in many practical situations (\ref{cond}) is easier to verify than (\ref{Davies}). For example, if $L$ is the Laplace Beltrami operator with Dirichlet or Neumann boundary condition, then (\ref{cond}) holds trivially, because $|\nabla \varphi(x)| \leq \frac{r}{t}$ (as the gradient of the distance function to any set is known as a 1-Lipschitz function), which gives us back the special case of finite propagation speed of $\text{cos }t\sqrt{-\Delta}$. Verifying (\ref{Davies}) seems to be harder in this case.
	
\end{remark}

Now, we extend the range of $H$ in Proposition \ref{FSOP} to $H \in L^2(M)$. Towards that end, pick $H_n$ continuous such that $H_n \longrightarrow H$ and consider $\Cal{L}_n$ given by $-\Cal{L}_n = D^*D + H_n$. Let $-\Cal{L} = D^*D + H$.\newline
%This is possible if $\text{Spec}(-L) \subset [\rho, \infty)$, with $\rho > 0$. One just has to choose $V$ such that $\inf V \geq -\rho$. For a fixed $y$, define $V_y(x) = \max\{V(x), -y\}$, and call $\Cal{L}_y = L + V_y$. \newline
We can see that $\Cal{L}_n (u) \longrightarrow \Cal{L}(u)$ for $u \in \Cal{D}(D^*D)$ as $n \longrightarrow \infty$. That means, $\Cal{L}_n \longrightarrow \Cal{L}$ in the strong resolvent sense as $n \longrightarrow \infty$ (see ~\cite{RS}, Theorem VIII.25(a)). Then, $\mbox{cos }tx$ and $ e^{-tx}$ being bounded continuous functions on $\RR$ for all $t > 0$, by Theorem VIII.20(b) of \cite{RS}, we have $\forall u \in \Cal{D}(D)$,
\begin{align*}
\mbox{cos }t\sqrt{-\Cal{L}_n}u & \longrightarrow \mbox{cos  }t\sqrt{-\Cal{L}}u,\\
e^{t\Cal{L}_n}u & \longrightarrow e^{t\Cal{L}}u.
\end{align*} 
%, which means that it is enough to argue the finite propagation speed property of $\Cal{L}_y$ for all $y \in \RR$.\newline

The finite propagation speed of $\mbox{cos  }t\sqrt{-\Cal{L}}$ now follows  by the Davies-Gaffney estimates: if for a fixed pair $U, V \subset \overline{M}$ of open sets, $L^2$ sections $u, v$ such that $\mbox{supp  }u \subset U$, $\mbox{supp  }v \subset V$, $r = \mbox{dist }(U, V)$, we have 
\beq
( e^{t\Cal{L}_n}u, v) \leq e^{-\frac{r^2}{4t}}\V u\V_{L^2}\V v\V_{L^2},
\eeq
then in the limit, we must have 
\beq
(e^{t\Cal{L}}u, v) \leq e^{-\frac{r^2}{4t}}\V u\V_{L^2}\V v\V_{L^2}.
\eeq
\subsection{Acknowledgements} I thank my Ph.D. advisor Michael Taylor for bringing this project to my notice. I also thank Jeremy Marzuola and Perry Harabin for discussions.
\\
\\
\bibliographystyle{plain}
\def\noopsort#1{}
%\begin{thebibliography}{11}

\end{document}